\documentclass[a4paper,12pt]{article}
\usepackage[utf8x]{inputenc}
\usepackage{amsfonts,amsthm,amssymb,mathrsfs,amsmath}
\usepackage{graphicx,graphics}
\usepackage{relsize}
\usepackage[all]{xy}
\usepackage{tikz}
\usepackage{color}
\usepackage{hyperref}

\newcommand{\ZZ}{\mathbb{Z}}
\newcommand{\NN}{\mathbb{N}}
\newcommand{\CC}{\mathbb{C}}
\newcommand{\RR}{\mathbb{R}}

\newcommand{\Ex}{\mathop{\mathbb{E}}}
\newcommand{\dprod}[2]{\left\langle #1,#2\right\rangle}
\newcommand{\norm}[1]{\left\lVert #1\right\rVert}
\newcommand{\proj}[1]{\left|#1\right\rangle\left\langle#1\right|}

\newtheorem{thm}{Theorem}[section]

\newtheorem{lem}[thm]{Lemma}

\numberwithin{equation}{section}

\DeclareMathOperator{\dist}{d}
\DeclareMathOperator{\tr}{Tr}

\title{Eigenvalue Statistics for higher rank Anderson model over Canopy tree}
 \author{Narayanan P.A. \\ The Institute of Mathematical Sciences \\ Taramani, Chennai 600113, India \\ email: panarayanan@imsc.res.in}
\date{\today}

\begin{document}

\maketitle
\begin{abstract}
This work is focused on the local eigenvalue statistics for the Anderson tight
binding model with non-rank-one perturbations over the canopy tree, at
large disorder.
On the Hilbert space $\ell^2(\mathcal{C})$, where $ \mathcal{C}  $ is
the canopy tree, the 
random operator we consider is 
$\Delta_{\mathcal{C}}+\sum_{y\in J}\omega_y P_y$, where
$\Delta_{\mathcal{C}}$ is the adjacency operator over the tree,
$\{\omega_y\}_{y\in J}$ are i.i.d real random variables following
some absolutely continuous distribution having a bounded density with compact support, and $P_y$ are
projections on $\ell^2(\{x\in\mathcal{C}: \dist(y,x)<m_0~\&~y\prec
x\})$. For this operator, we show that, the eigenvalue-counting point process converges to compound Poisson process.
\end{abstract}
\tableofcontents
\section{Introduction}
In the theory of disordered systems, Anderson tight binding model is well studied for its spectral and dynamical properties. 
The spectral theory for the Anderson tight binding Hamiltonian over Bethe lattice has a rich structure and is one of the models where the existence of both the absolute continuous~\cite{ASW,FHS,KA1} as well as the pure point~\cite{AM,ASFH,FMSS} spectrum are proven.
Naturally, the next question is about the local structure of the spectrum, and so the eigenvalue statistics is an important question to study. 
But, the eigenvalue statistics as defined by Minami~\cite{M2} does not provide the eigenvalue statistics over the Bethe lattice, but over the canopy tree (as explained by Aizenman-Warzel~\cite{AW}). 
The main focus of this manuscript is to study the local eigenvalue
statistics for the Anderson tight binding model over the canopy tree when single site potential affects a collection of vertices of the tree.
Although, to define the point process we look at the cut-off operator on the Bethe lattice.

\noindent
\begin{minipage}{0.6\textwidth}
~\\
To describe our main result we need to set up a few notations first. Let $\mathcal{B}=(V_{\mathcal{B}},E_{\mathcal{B}})$ denote the infinite rooted tree with the root $0\in V_{\mathcal{B}}$; all the vertices has $K+1$ neighbours (in the figure, $K$ is $2$). 
On the Hilbert space $\ell^2(\mathcal{B})$, we have the graph Laplacian $\Delta$ defined by
$$(\Delta\psi)(x)=\sum_{\dist(x,y)=1}\psi(y),~~\forall x\in
V_{\mathcal{B}},\psi\in\ell^2(\mathcal{B}).$$
Here, $ \dist(x,y)  $ is the usual distance on the graphs,
which is the length (i.e., the number of edges) of the shortest path
between the vertices $ x $ and $ y $.
The higher ranked Anderson type operator on the Bethe lattice $\mathcal{B}$ is defined as
\begin{equation}\label{AndOpBetheEq1}
 H^\omega_\lambda:=\Delta+\lambda\sum_{y \in J} \omega_y P_y,
\end{equation}
\end{minipage}
\begin{minipage}{0.4\textwidth}
\begin{flushright}
  \includegraphics[width=2in,keepaspectratio]{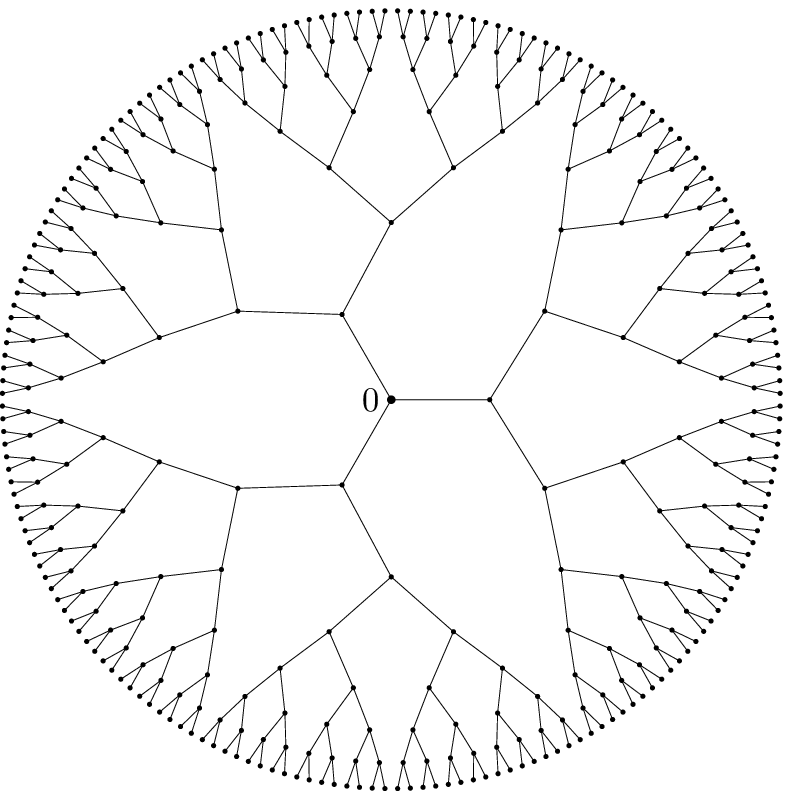}
    \end{flushright}
\end{minipage}
where $\lambda>0$ is the disorder parameter, $\{\omega_y\}_{y\in J}$ are independent identically distributed  real random variables following absolutely continuous distribution $\rho(x)dx$ where $\rho\in L^\infty(\RR)$ and $supp(\rho)$ compact. The projections $P_y$ are defined to be
\begin{equation}\label{defPot}
 (P_y \psi)(x)=\left\{\begin{matrix} \psi(x), &  \dist(y,x)\leq m_0
     ~\&~ y\prec x,\\ 0, & \text{other wise},\end{matrix}\right.
\end{equation}
for $y\in J$. Note that, $rank(P_y)=\frac{K^{m_0+1}-1}{K-1}$ which we
will denote by $M_0$. Here, $y\prec x$ means that the vertex $x$ is such
that $\dist(0,x)=\dist(0,y)+\dist(y,x)$; i.e., $y$ lies between $0$
and $x$ (equivalently, $ x $ is forward to $ y $). Finally, the indexing set $J$ is defined by
\begin{equation}\label{IndSetEq1}
 J=\{x\in V_{\mathcal{B}}: \dist(0,x)\in (m_0+1)\NN\cup\{0\}\}.
\end{equation}
Since our main concern is to study the eigenvalue process, we will work with the cut-off operator 
\begin{equation}\label{mainOp1}
H^\omega_{\lambda,L} = \chi_{\Lambda_L(0)}H^\omega_\lambda\chi_{\Lambda_L(0)}
\end{equation}
on $\ell^2(\Lambda_L(0))$, where 
$$\Lambda_L(x)=\{y\in V_{\mathcal{B}}:~\dist(x,y)\leq L \},$$
and the projection $\chi_U$, for  $U\subseteq V_{\mathcal{B}}$, is defined as 
$$(\chi_{U}\psi)(x)=\left\{\begin{matrix} \psi(x), & x\in U,\\ 0, &
    \text{other wise},\end{matrix}\right.\qquad\forall \psi\in\ell^2(\mathcal{B}).$$
But from now onwards, for convenience, we will denote $ \Lambda_{L}(0)  $ by $ \Lambda_{L} $.

To study the local eigenvalue statistics at $E_0\in\RR$, we will look at the limit of random point processes $\{\mu^{\omega,\lambda}_{E_0,L}\}_{L\in\NN}$ defined by
\begin{equation}\label{meaSeqEq1}
 \mu^{\omega,\lambda}_{E_0,L}(f)=\tr(f(|\Lambda_L|(H^\omega_{\lambda,L}-E_0))),~\forall f\in C_c(\RR),
\end{equation}
where $ C_{c}(\RR)  $ is  the set of all continuous functions with
compact support on $ \RR $.
As stated earlier, this method of defining the point process does not
provide local eigenvalue statistics over the Bethe lattice, but on the canopy tree. 
 The canopy tree $\mathcal{C}=(V_{\mathcal{C}},E_{\mathcal{C}})$ is defined recursively, layer by layer, starting from the boundary vertices, $\mathcal{C}_0=\partial\mathcal{C}$ (a countable set of vertices). 
 Each layer $\mathcal{C}_n$ (a countable collection of vertices) is partitioned into sets of $K$ vertices, which are joined to a single unique vertex in the layer $\mathcal{C}_{n+1}$. 
 Notice that, in the graph defined through this process, for any $x\in
\mathcal{C}_n$, we have, $\dist(x,\partial \mathcal{C})=n$. [See figure~[\ref{fig:image2}]].
\begin{figure}[ht]
\begin{center}
\includegraphics[height=1in,keepaspectratio]{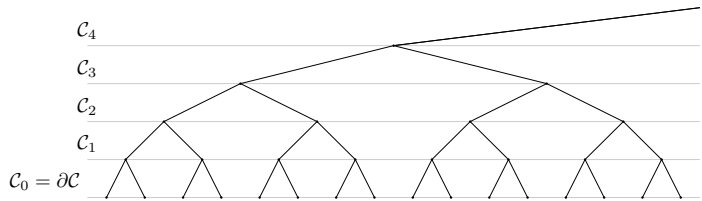}   
\end{center}
\caption{First few recursion steps for the canopy tree for $K=2$.}
\label{fig:image2}
\end{figure}
On the canopy graph we have the random operator
\begin{equation}\label{AndOpCanopyEq1}
H^\omega_{\mathcal{C},\lambda}=\Delta_{\mathcal{C}}+\lambda\sum_{y\in J_{\mathcal{C}}}\omega_y P_y,
\end{equation}
where $P_y:=\chi_{\tilde{\Lambda}_{m_0}(y)}$ for $y\in J_{\mathcal{C}}$,  $\lambda>0$, is the disorder parameter, and $\{\omega_y\}_{y\in J_{\mathcal{C}}}$ are i.i.d real random variables following the distribution $\rho(x)dx$. Here,
$$J_{\mathcal{C}}:=\{y\in V_{\mathcal{C}}: \dist(\partial
\mathcal{C},y)=m_0+(m_0+1)k,~\text{for some}~ k\in\NN\cup\{0\}\},$$
and 
$$\tilde{\Lambda}_{m_0}(y)=\{x\in V_{\mathcal{C}}: \dist(y,x)\leq m_0~\&~ \dist(\partial\mathcal{C},y)=\dist(\partial\mathcal{C},x)+\dist(x,y)\}.$$
Note that, removing the root of $ \Lambda_{L} $ of the Bethe lattice
we are left with a collection of $ K + 1  $ sub-trees each of which
can be identified with a sub-tree $ \tilde{\Lambda}_{L -1}(y)  $ of the canopy
tree, for some $ y $ such that $ \dist(y,\partial\mathcal{C}) = L - 1
$. Intuitively,
from the perspective of the root of the sub-tree $ \Lambda_L $, as $ L
\rightarrow \infty $ it describes the Bethe lattice; but, from the
perspective of the vertices near the boundary (in other words, the canopy) of $ \Lambda_L $, it describes the canopy tree.

With these definitions in place we have:
\begin{thm}\label{LocResThm}
 Let $H^\omega_{\lambda,L}$ be defined as in~\eqref{mainOp1}. Then, for any $0<s<1$, $E_0\in\RR$, and $\gamma>0$, there exist $\lambda_{\gamma,s}>0$ and $C>0$ such that,
\begin{equation}\label{exp}
 \sup_{\epsilon>0}\Ex_\omega\left[\left|\dprod{\delta_x}{\left(H^\omega_{\lambda,L}-E-i \epsilon\right)^{-1}\delta_y}\right|^s\right]\leq Ce^{-\gamma \dist(x,y)}
\end{equation}
for all $\lambda>\lambda_{\gamma,s}$ and $L$ large enough so that $x,y\in\Lambda_L$.
\end{thm}
The above theorem describes the exponential decay of the Green's
function. But, what is more important is the fact that any rate of
decay is achievable by changing the disorder parameter. The next
theorem is about the regularity of the density of states for the model.
\begin{thm}\label{ThmDosCanopy}
For any interval $I\subset\RR$, we have,
\begin{equation}\label{DosConvEq1}
 \frac{1}{|\Lambda_L|}\tr(E_{H^\omega_{\lambda,L}}(I))\xrightarrow{L\rightarrow\infty}n_{\mathcal{C},\lambda}(I)\qquad a.s.,
\end{equation}
where
\begin{equation}\label{dosDefEq1}
 n_{\mathcal{C},\lambda}(I)=\frac{K-1}{K}\sum_{n=0}^\infty K^{-n}\Ex_\omega\left[\dprod{\delta_{x_n}}{E_{H^\omega_{\mathcal{C},\lambda}}(I)\delta_{x_n}}\right].
\end{equation}
Here, $\{x_n\}_{n=0}^\infty$ is a sequence of vertices of $V_{\mathcal{C}}$, such that, $\dist(\partial\mathcal{C},x_n)=n$. The measure $n_{\mathcal{C},\lambda}$ is absolutely continuous w.r.t the Lebesgue measure.
\end{thm}
With the above theorems in place, we can state our main result:
\begin{thm}\label{mainThm}
  For any $E_0\in\RR$, and $ \lambda > 0 $ large enough, define the sequence of measures $\{\mu^{\omega,\lambda}_{E_0,L}\}_{L\in\NN}$ by~\eqref{meaSeqEq1}. 
 For any bounded interval $I$, there exists a sequence of natural numbers $\{L_n\}_{n\in\NN}$ such that, the random variables $\{\mu^{\omega,\lambda}_{E_0,L_n}(I)\}_n$ converge to $\mathcal{P}_I^\omega$, a compound Poisson random variable, in the sense of distribution. 
 The characteristic function $\Ex[e^{\iota t\mathcal{P}^\omega_I}]$ is of the form $e^{\sum_{k=1}^{M_0} (e^{\iota tk}-1) p_k(I)}$ with the property  $p_k(I)\leq \frac{K n_{\mathcal{C},\lambda}(E_0)}{k}|I|$ for all $1\leq k\leq M_0$.
\end{thm}

It should be noted that the operators $H^\omega_{\mathcal{C},\lambda,L}$ and $H^\omega_{\lambda,L}$ can have non-trivial multiplicity. 
This is because of the fact that any symmetry of the tail
sub-trees ($\tilde{\Lambda}_{m_0}(y)$ for
$\dist(y,\partial\mathcal{C})=m_0$) produces a unitary operator which
fixes $H^\omega_{\mathcal{C},\lambda}$ (similar thing happens in the case of $H^\omega_{\lambda,L}$).

The eigenvalue statistics in one dimension was studied by Molchanov~\cite{M1}, and later for higher dimensions by Minami~\cite{M2}. 
In the region of fractional localization (where~\eqref{exp} holds), they showed that the statistics is Poisson. 
Subsequently, the Poisson statistics was shown for the trees by Aizenman-Warzel~\cite{AW}, and by Geisinger~\cite{GL} for regular graphs. 
In some recent results, Germinet-Klopp~\cite{GK} extended the results of Killipp-Nakano~\cite{NK}. These works are focused on eigenfunction statistics in the regime of pure point spectrum.
An analogue of Minami's~\cite{M2} work was done by Dolai-Krishna~\cite{DK2}, with $\alpha$-H\"{o}lder continuous single site distribution. 
There are also works in the region of absolutely continuous spectrum, like Kotani-Nakano~\cite{KN}, Avila-Last-Simon~\cite{ALS}, and Dolai-Mallick~\cite{MD1}.
There are a few results for spectral statistics for non-rank one case, for example, Hislop-Krishna~\cite{HK1} and Combes-Germinet-Klein~\cite{CGK1}.

This work is inclined towards  the works of Aizenman-Warzel~\cite{AW} and Hislop-Krishna~\cite{HK1}. 
In the work~\cite{AW}, the authors concluded simple Poisson point
process as the eigenvalue statistics for the Anderson tight binding model over the canopy tree. 
One of the important points they raised is the fact that infinite divisibility of the eigenvalue process cannot be taken similar to $\ZZ^d$ case. 
This is because $\frac{|\partial\Lambda_L|}{|\Lambda_L|}$ does not converges to zero as $L\rightarrow\infty$.
But, because of the exponential nature of the growth of the surface area, and the fact that we can achieve any rate  of decay in Theorem~\ref{LocResThm}, we can get the infinite divisibility needed for Poisson process by dividing the trees into sub-trees of height $\approx \alpha L$ (for $\alpha>0$ small enough). 
Usually, this would fail to produce the correct decay needed to
establish the infinite divisibility; but in this case, this   is enough.
\section{Preliminaries}
In this section, some important results are established which are
essential for proving the main results.  Before that, a few notations are needed.
For $y\in\Lambda_L$, we will denote 
\begin{equation}\label{defForBoxEq1}
 \Lambda^\prime_l(y):=\{x\in\Lambda_L: \dist(x,y)\leq l~\&~y\prec x\},
\end{equation}
for $l\in\NN$. Therefore, for any $p\in J$,  the projections $P_p=\chi_{\Lambda^\prime_{m_0}(p)}$. Using the resolvent equation between $H^\omega_{\lambda,L}$ and
$$\tilde{H}^\omega_L:=(I-P_p)H^\omega_{\lambda,L} (I-P_p)+P_p\Delta P_p+\omega_p P_p,$$
we have,
\begin{align}
&\dprod{\delta_x}{(H^\omega_{\lambda,L}-z)^{-1}\delta_y}=\nonumber\\
&-\left\langle\delta_x,\left[P_p\Delta P_p+(\omega_p-z)P_p-P_p\Delta(\chi_{\Lambda_L}-P_p)(\tilde{H}^\omega_L-z)^{-1}(\chi_{\Lambda_L}-P_p)\Delta P_p\right]^{-1}\right.\nonumber\\
&\qquad\qquad\qquad\qquad\qquad\qquad\qquad\left. P_p\Delta(\chi_{\Lambda_L}-P_p)(\tilde{H}^\omega_L-z)^{-1}\delta_y\right\rangle\label{HertlotzMatEq2}
\end{align}
for $x\in \Lambda_{m_0}^\prime(p)$ and $y\in\Lambda_L\setminus\Lambda_{m_0}^\prime(p)$. By taking $y\in\Lambda^\prime_{m_0}(p)$, we can also show that,
\begin{align}\label{HertlotzMatEq1}
  \begin{split}
    &P_p(H^\omega_{\lambda,L}-z)^{-1}P_p= \\
    &\hspace{1cm}\left[P_p\Delta
      P_p+(\omega_p-z)P_p-P_p\Delta(\chi_{\Lambda_L}-P_p)(\tilde{H}^\omega_L-z)^{-1}(\chi_{\Lambda_L}-P_p)\Delta
      P_p\right]^{-1}.
  \end{split}
\end{align}
Using the fact that there is a unique path from $x$ to $y$, (in the sense that if we remove any edge within this path, then $x$ and $y$ will lie in different components) say, $x=x_1,\ldots,x_n=y$, and taking $n_1<n$ so that $x_{n_1}\in\Lambda^\prime_{m_0}(p)$ and $x_{n_1+1}\in\Lambda_L\setminus\Lambda^\prime_{m_0}(p)$, the expression~\eqref{HertlotzMatEq2} gives us
\begin{equation}\label{GreenFunEq1}
 \dprod{\delta_x}{(H^\omega_{\lambda,L}-z)^{-1}\delta_y}=\Gamma_{x_1,x_{n_1}}\dprod{\delta_{x_{n_1+1}}}{(\tilde{H}^\omega_L-z)^{-1}\delta_y},
\end{equation}
where $\Gamma_{a,b}$ is
\begin{equation}\label{GreenFunEq3}
\dprod{\delta_a}{\left[P_p\Delta P_p+(\omega_p-z)P_p-P_p\Delta(\chi_{\Lambda_L}-P_p)(\tilde{H}^\omega_L-z)^{-1}(\chi_{\Lambda_L}-P_p)\Delta P_p\right]^{-1}\delta_b}
\end{equation}
for $a,b\in\Lambda^\prime_{m_0}(p)$. Repeating this procedure inductively, we have,
\begin{equation}\label{GreenFunEq2}
 \dprod{\delta_x}{(H^\omega_{\lambda,L}-z)^{-1}\delta_y}=\prod_{i=0}^m \Gamma_{x_{n_i+1},x_{n_{i+1}}},
\end{equation}
with $x=x_1,\ldots,x_n=y$  the shortest path between $x$ and $y$, and $\{n_i\}_{i=1}^m$ with the property that for each $i$ there exists $p_i\in J$ such that $x_{n_{i-1}+1},x_{n_i}\in\Lambda^\prime_{m_0}(p_i)$,  $n_0=0$, and $n_{m+1}=n$.
Finally,
\begin{align}\label{GreenFunEq4}
&\Gamma_{x_{n_i+1},x_{n_{i+1}}}=\Bigg\langle\delta_{x_{n_i+1}},\Bigg[P_{p_i}\Delta P_{p_i}+(\omega_{p_i}-z)P_{p_i}\\
&\left.\left.-P_{p_i}\Delta\left(\chi_{\Lambda_L}-\sum_{j=1}^i P_{p_j}\right)(\tilde{H}^\omega_{i,L}-z)^{-1}\left(\chi_{\Lambda_L}-\sum_{j=1}^iP_{p_j}\right)\Delta P_{p_i}\right]^{-1}\delta_{x_{n_{i+1}}}\right\rangle, \nonumber
\end{align}
\noindent
\begin{minipage}{0.8\textwidth}
 where
$$\tilde{H}^\omega_{i,L}:=(\chi_{\Lambda_L}-P_{p_i})H^\omega_{i-1,L}
(\chi_{\Lambda_L}-P_{p_i})+P_{p_i}\Delta P_{p_i}+\omega_{p_i} P_{p_i}$$
with $\tilde{H}^\omega_{0,L}:=H^\omega_{\lambda,L}$.
Observe that, 
$$P_{p}\Delta\left(\chi_{\Lambda_L}-P_{p}\right)(\tilde{H}^\omega_{L}-z)^{-1}\left(\chi_{\Lambda_L}-P_{p}\right)\Delta P_{p}$$ 
is multiplication operator over the boundary of the sub-tree $\Lambda^\prime_{m_0}(p)$. After removing the sub-tree $\Lambda^\prime_{m_0}(p)$, we are left with disjoint trees, and $\tilde{H}^\omega_L$ restricted on each of these sub-trees are independent of each other. For  $y\in\Lambda^\prime_{m_0}(p)$, define
$$N_y=\{x\in\Lambda_L: \dist(x,y)=1~\&~x\not\in\Lambda^\prime_{m_0}(p) \},$$
which is the set of neighbours of the vertex $y$, which lie outside $
\Lambda'_{m_{0}(p) }  $. We have,
\begin{align*}
&P_{p}\Delta\left(\chi_{\Lambda_L}-P_{p}\right)(\tilde{H}^\omega_{L}-z)^{-1}\left(\chi_{\Lambda_L}-P_{p}\right)\Delta P_{p}\\
 &\qquad=\sum_{y\in\Lambda^\prime_{m_0}(p)} \proj{\delta_y}\sum_{x\in N_y}\dprod{\delta_x}{(\tilde{H}^\omega_{L}-z)^{-1}\delta_x}
\end{align*}
and the independence of $\tilde{H}^\omega_L$ on each of the subtree
implies the independence of $\left\{\dprod{\delta_x}{(\tilde{H}^\omega_{L}-z)^{-1}\delta_x}\right\}_x$ for each $x\in \cup_{y\in\Lambda^\prime_{m_0}(p)} N_y$.
\end{minipage}
\begin{minipage}{0.2\textwidth}
\begin{flushright}
\includegraphics[height=3.5in,keepaspectratio]{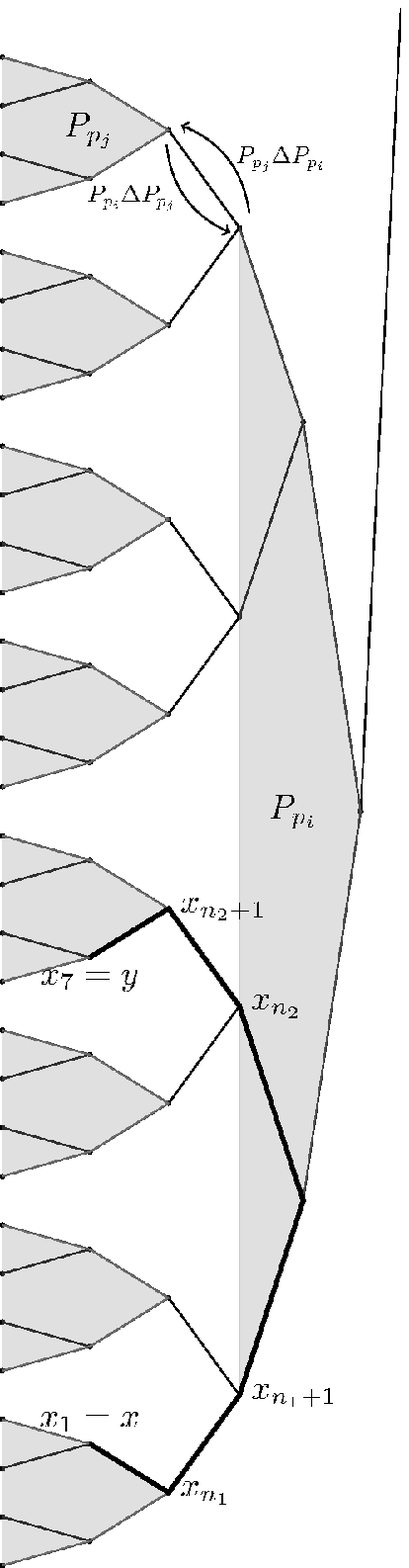} 
\end{flushright}
\end{minipage}

With these notations, we are ready to establish the Wegner and the
Minami estimates. Notice that, $rank(P_p)$, (which we have called as $
M_{0}  $) is same as $|\Lambda_{m_0}^\prime(p)|$.
Even though there are multiple proofs of the Wegner estimate, for example~\cite{CGK,CHK,CH}, those proofs are in more general settings, and use more sophisticated techniques. 
In the case of projection valued perturbations, the proof can be done
using rank one case as the basis, as done here.
\begin{lem}\label{lemWegner}{\bf (Wegner Estimate)}
 For any bounded interval $I\subset\RR$, we have,
 \begin{align}
  &\Ex_\omega\left[\dprod{\delta_x}{E_{H^\omega_{\lambda,L}}(I)\delta_x}\right]\leq
   C|I|,\label{WegnerEq1} \text{and}\\
  & \Ex_\omega[\tr(E_{H^\omega_{\lambda,L}}(I))]\leq C|I||\Lambda_L|\label{WegnerEq2}.
 \end{align}
\end{lem}
\begin{proof}
The proof follows  similar steps as in the rank one case. Using Stone's formula~\cite[Theorem VII.13]{RS}, we have,
$$\frac{1}{2}\left(E_{H^\omega_{\lambda,L}}(I^{cls})+E_{H^\omega_{\lambda,L}}(I^{int})\right)=\text{s-}\lim_{\epsilon\rightarrow 0}\int_I \Im (H^\omega_{\lambda,L}-E-\iota\epsilon)^{-1}dE,$$
(here, $ I^{cls}  $ and $ I^{int}  $ are the closure and interior of $
I $, respectively) and since $\Im (H^\omega_{\lambda,L}-E-\iota\epsilon)^{-1}$ is non-negative definite, we can use Tonelli's theorem~\cite[Theorem 3.7.7]{BC1} to get
\begin{align}
&\Ex_\omega\left[\dprod{\delta_x}{\frac{1}{2}\left(E_{H^\omega_{\lambda,L}}(I^{cls})+E_{H^\omega_{\lambda,L}}(I^{int})\right)\delta_x}\right]\nonumber\\
&\qquad=\lim_{\epsilon\rightarrow0}\Ex_\omega\left[\int_I \dprod{\delta_x}{\Im (H^\omega_{\lambda,L}-E-\iota\epsilon)^{-1}\delta_x}dE\right]\nonumber\\
&\qquad=\lim_{\epsilon\rightarrow0}\int_I \Ex_\omega[\dprod{\delta_x}{\Im (H^\omega_{\lambda,L}-E-\iota\epsilon)^{-1}\delta_x}]dE\label{wegPfEq1}.
\end{align}
So, to get~\eqref{WegnerEq1} and~\eqref{WegnerEq2}, we need to estimate
$$\Ex_\omega[\dprod{\delta_x}{\Im (H^\omega_{\lambda,L}-E-\iota\epsilon)^{-1}\delta_x}]$$
independent of $E$ and $\epsilon$.

We can re-write~\eqref{HertlotzMatEq1},as
$$P_p(H^\omega_{\lambda,L}-z)^{-1}P_p=\left[\omega_pI - A^\omega(z)\right]^{-1},\quad\forall z\in\CC^{+},$$
where we have collected all the terms occurring in~\eqref{HertlotzMatEq1} other than $ \omega_{p}  $
into $ A^{\omega }(z)  $. We can see that $ A^{\omega }(z)  $ doesn't
depend on $ \omega_{p}  $. 
Now, let $\{E_i^{\tilde{\omega},p,z}\}_{i}$ denote the eigenvalues of
the matrix $A^\omega(z)$. Because
$P_p(H^\omega_{\lambda,L}-z)^{-1}P_p$ is a matrix valued Herglotz
function, one can see that, $A^\omega(z)$ is also a matrix valued
Herglotz function, and so all the eigenvalues have positive imaginary
part.
Hence,
\begin{align}
\Ex_\omega[\tr(\Im P_y(H^\omega_{\lambda,L}-z)^{-1}P_y)]&=\Ex_{\tilde{\omega}}\left[\int \sum_i \Im \frac{1}{x-E_i^{\tilde{\omega},p,z}}\rho(x)dx\right]\nonumber\\
&\leq \norm{\rho}_\infty\sum_i\Ex_{\tilde{\omega}}\left[\int \frac{\Im E_i^{\tilde{\omega},p,z}}{(x-\Re E_i^{\tilde{\omega},p,z})^2+(E_i^{\tilde{\omega},p,z})^2}dx\right]\nonumber\\
&\leq \pi m_0\norm{\rho}_\infty.\label{WegnerEq3}
\end{align}
Using the above estimate we get,
\begin{align*}
 \Ex_\omega[\tr(E_{H^\omega_{\lambda,L}}(I))]&=\sum_{y\in J}\Ex_\omega[\tr(P_y E_{H^\omega_{\lambda,L}}(I)P_y)]\\
 &=\sum_{y\in J}\lim_{\epsilon\rightarrow0}\int_I \Ex_\omega[\tr(P_y(H^\omega_{\lambda,L}-E-\iota\epsilon)^{-1}P_y)]dE\\
 &\leq \norm{\rho}_\infty \pi |\Lambda_L||I|.
\end{align*}

\end{proof}

Since, the model we are concerned with involves higher rank
perturbations, it is possible that operators in our model might have
eigenvalues of multiplicity greater than one. Therefore, in general we
might not be able to get a proper Minami estimate. Below, we prove
an extended version of the Minami estimate.

\begin{lem}\label{lemMinami}{\bf (Extended Minami Estimate)}
 For any bounded interval $I\subset\RR$, we have,
\begin{equation}\label{MinamiEq1}
 \sum_{m\geq M_0}\mathbb{P}[\tr(E_{H^\omega_{\lambda,L}}(I))>m]\leq (\pi\norm{\rho}_\infty |\Lambda_L||I|)^2,
\end{equation}
where $ M_{0}  $ is the common rank of the perturbing projections.
\end{lem}
\begin{proof}
With out loss of generality assume
$supp(\rho)\subseteq[a,b]$. Following the notations from the previous
lemma, for any $y\in J$, we have,
$$H^\omega_{\lambda,L}\leq H^{\tilde{\omega}}_L+(b+\lambda)P_y\qquad\forall \lambda>0.$$
Using
$$|\tr(E_{H^\omega_{\lambda,L}}(I))-\tr(E_{H^{\tilde{\omega}}_L+(b+\lambda)P_y}(I))|\leq M_0,$$
we have,
$$\tr(E_{H^{\omega}_L}(I))-M_0\leq \tr(E_{H^{\tilde{\omega}}_L+(b+\lambda)P_y}(I)),$$
and in particular,
$$\tr(E_{H^{\omega}_L}(I))-M_0\leq \int \tr(E_{H^{\tilde{\omega}}_L+(b+\lambda)P_y}(I))\rho(\lambda+a)d\lambda.$$
So, we can use the above and get,
\begin{align*}
&\sum_{m\geq M_0}\mathbb{P}[\tr(E_{H^\omega_{\lambda,L}}(I))>m]\\
&\leq \Ex_\omega\left[\tr(E_{H^\omega_{\lambda,L}}(I))(\tr(E_{H^\omega_{\lambda,L}}(I))-M_0)\chi(\tr(E_{H^\omega_{\lambda,L}}(I))>M_0)\right]\\
&=\sum_{y\in J}\Ex_\omega\left[\tr(P_y E_{H^\omega_{\lambda,L}}(I)P_y)(\tr(E_{H^\omega_{\lambda,L}}(I))-M_0)\chi(\tr(E_{H^\omega_{\lambda,L}}(I))>M_0)\right]\\
&\leq \sum_{y\in J}\Ex_\omega\left[\tr(P_y E_{H^\omega_{\lambda,L}}(I)P_y)\tr(E_{H^{\tilde{\omega}}_L+(b+\lambda_y)P_y}(I))\chi(\tr(E_{H^\omega_{\lambda,L}}(I))>M_0)\right]\\
&\leq \sum_{y\in J}\Ex_\omega\left[\tr(P_y E_{H^\omega_{\lambda,L}}(I)P_y)\int \tr(E_{H^{\tilde{\omega}}_L+(b+\lambda_y)P_y}(I))\rho(\lambda_y+a)d\lambda_y\right]\\
  \begin{split}
   &\leq \sum_{y\in J}\Ex_{\tilde{\omega}}\biggl[\left(\int
        \tr(E_{H^{\tilde{\omega}}_L+(b+\lambda_y)P_y}(I))\rho(\lambda_y+a)d\lambda_y\right) \biggr.\\
      &\hspace{3cm}  \bigg.\left(\int_a^b \tr(P_y E_{H^{\tilde{\omega}}_L+x P_y}(I)P_y)\rho(x)dx\right)
    \bigg]
  \end{split}\\
&\leq \pi\norm{\rho}_\infty m_0 |I|\sum_{y\in J}\Ex_{\tilde{\omega}}\left[\int \tr(E_{H^{\tilde{\omega}}_L+(b+\lambda_y)P_y}(I))\rho(\lambda_y+a)d\lambda_y\right].
\end{align*}
So, taking $\{\lambda_y\}_{y\in J}$ to be i.i.d random variables
following the distribution $\rho(x+a)dx$, independent of
$\{\omega_y\}_{y\in J}$, we can use the Wegner
estimate~\eqref{WegnerEq2}, to get~\eqref{MinamiEq1}.

\end{proof}
\section{Results}

\subsection{Proof of Theorem~\ref{LocResThm}}
To prove the theorem we will use the expression~\eqref{GreenFunEq2}. Notice that, in that expression, $\Gamma_{x_{n_i+1},x_{n_{i+1}}}$ is independent of the random variables $\{\omega_{p_j}\}_{j=1}^{i-1}$. So we have,
\begin{align*}
  \Ex_\omega\left[\left|\dprod{\delta_x}{(H^\omega_{\lambda,L}-z)^{-1}\delta_y}\right|^s\right]&=\\
&\hspace{-1.8cm}\Ex_{\omega_{p_0}^{\perp},\cdots,\omega_{p_m}^{\perp}}\left[\Ex_{\omega_{p_0}}\left[|\Gamma_{x_{n_0+1},x_{n_{1}}}|^s\Ex_{\omega_{p_1}}\left[\cdots\Ex_{\omega_{p_m}}\left[|\Gamma_{x_{n_m+1},x_{n_{m+1}}}|^s\right]\right] \right]\right].
\end{align*}
Therefore, all we need to do is to estimate $\Ex_{\omega_{p_i}}\left[|\Gamma_{x_{n_i+1},x_{n_{i+1}}}|^s\right]$ independent of $\{\omega_n\}_{n\neq p_i}$. Let $\{E_j^{\omega}(z)\}_{j=1}^{rank(P_{p_i})}$, counted with multiplicity, denote the eigenvalues of 
$$P_{p_i}\Delta P_{p_i}-P_{p_i}\Delta\left(\chi_{\Lambda_L}-\sum_{j=1}^i P_{p_j}\right)(\tilde{H}^\omega_{i,L}-z)^{-1}\left(\chi_{\Lambda_L}-\sum_{j=1}^iP_{p_j}\right)\Delta P_{p_i}.$$
Then, by the definition of $\Gamma$ (see~\eqref{GreenFunEq4}), we have,
$$|\Gamma_{x_{n_i+1},x_{n_{i+1}}}|\leq
\sum_{j=1}^{rank(P_{p_i})}\frac{1}{|E^\omega_i(z)-\lambda\omega_{p_i}-z|}\
.$$
Hence,
\begin{align*}
 \Ex_{\omega_{p_i}}\left[|\Gamma_{x_{n_i+1},x_{n_{i+1}}}|^s\right]&\leq \Ex_{\omega_{p_i}}\Biggl[ \sum_{j=1}^{rank(P_{p_i})}\frac{1}{|E^\omega_i(z)-\lambda\omega_{p_i}-z|^s}\Biggr]\\
 &\leq C\frac{|\Lambda^\prime_{m_0}(p_i)|}{\lambda^s}.
\end{align*}
Therefore, for large enough $\lambda$,  $C\frac{|\Lambda^\prime_{m_0}(p_i)|}{\lambda^s}<1$. So, using
\begin{equation}\label{expResEq1}
 \Ex_\omega\left[\left|\dprod{\delta_x}{(H^\omega_{\lambda,L}-z)^{-1}\delta_y}\right|^s\right]\leq \left(C\frac{|\Lambda^\prime_{m_0}(p_i)|}{\lambda^s}\right)^m,
\end{equation}
we get the estimate~\eqref{exp}, proving the theorem.

\subsection{Proof of Theorem~\ref{ThmDosCanopy}}
To show~\eqref{DosConvEq1}, it is enough to show 
$$\lim_{L\rightarrow\infty} \frac{1}{|\Lambda_L|}\tr(f(H^\omega_{\lambda,L}))=\int f(x)dn_{\mathcal{C},\lambda}(x),$$
for $f\in C_0(\RR)$. (For us it is enough to show the above result for functions in $C_c(\RR)$; but since $C_c(\RR)$ is contained in $C_0(\RR)$, clearly this suffices.)  Notice that
$$\frac{1}{|\Lambda_L|}\tr(f(H^\omega_{\lambda,L}))=\frac{1}{|\Lambda_L|}\sum_{r=0}^L\sum_{\dist(0,x)=L-r}\dprod{\delta_x}{f(H^\omega_{\lambda,L})\delta_x},$$
which can be written as
\begin{align*}
\begin{split}
&\frac{1}{|\Lambda_L|}\tr(f(H^\omega_{\lambda,L}))= \\ & \hspace{1.5cm}\frac{1}{1+(K+1)\frac{K^L-1}{K-1}}\biggl(\dprod{\delta_0}{f(H^\omega_{\lambda,L})\delta_0}+\sum_{r=0}^{L-1}
    (K+1)K^{L-r-1} C^\omega_{L,r}(f)\biggr),
\end{split}
\end{align*}
where
$$C^\omega_{L,r}(f):=\frac{1}{(K+1)K^{L-r-1}}\sum_{\dist(0,x)=L-r}\dprod{\delta_x}{f(H^\omega_{\lambda,L})\delta_x},$$
that is, the sum is done over the set of vertices  which are at a
distance  $r$ from the boundary. Notice that $|C^\omega_{L,r}(f)|\leq \|f\|_\infty$. Hence, It is enough to show 
$$\lim_{L\rightarrow\infty }C^\omega_{L,r}(f)=\Ex_\omega[\dprod{\delta_{x_r}}{f(H^\omega_{\mathcal{C},\lambda})\delta_{x_r}}],$$
where $x_r\in V_{\mathcal{C}}$ is such that $\dist(x_r,\partial \mathcal{C})=r$. Since $\Im(\cdot-z)^{-1}$ for $z\in\CC^{+}$ are dense in $C_0(\RR)$, to prove the above expression, it is enough to show 
$$\lim_{L\rightarrow\infty }C^\omega_{L,r}(\Im(\cdot-z)^{-1})=\Ex_\omega[\Im \dprod{\delta_{x_r}}{(H^\omega_{\mathcal{C},\lambda}-z)^{-1}\delta_{x_r}}].$$
For $L_0\in\NN$, define
$$S_{L-r-L_0}:=\{x\in\Lambda_L: \dist(x,\Lambda_{L-r-L_0}(0))=1\},$$
and for $y\in S_{L-r-L_0}$, set
$$\Lambda_{L,y}:=\{x\in\Lambda_L:y\prec x\}.$$
Then, using the resolvent identity with $H^\omega_{\lambda,L}$ and
$$\chi_{\Lambda_{L-r-L_0}(0)} H^\omega_{\lambda,L} \chi_{\Lambda_{L-r-L_0}(0)}+\sum_{x\in S_{L-r-L_0}} \chi_{\Lambda_{L,x}} H^\omega_{\lambda,L} \chi_{\Lambda_{L,x}},$$
for $x\in\Lambda_L$ such that $\dist(x,0)=L-r$, we get,
\begin{align*}
&\dprod{\delta_{x}}{(H^\omega_{\lambda,L}-z)^{-1}\delta_{x}}\\
&=\dprod{\delta_{x}}{(H^\omega_{\lambda,L,A_x}-z)^{-1}\delta_{x}}-\dprod{\delta_x}{(H^\omega_{\lambda,L}-z)^{-1}\delta_{PA_x}}\dprod{\delta_{A_x}}{(H^\omega_{\lambda,L,A_x}-z)^{-1}\delta_x},
\end{align*}
where $A_x$ is the unique  vertex in $S_{L-r-L_0}$ such that $x\in\Lambda_{L,A_x}$, $H^\omega_{\lambda,L,x}:=\chi_{\Lambda_{L,x}} H^\omega_{\lambda,L} \chi_{\Lambda_{L,x}}$ and $PA_x\in\Lambda_{L-r-L_0}$ such that $\dist(A_x,PA_x)=1$. 
We will choose $L_0$ so that $H^\omega_{\lambda,L,x}$ are i.i.d for each $x\in S_{L-r-L_0}$. Hence,
\begin{align*}
\biggl|(K+1)K^{L-r-1} C^\omega_{L,r}((\cdot-z)^{-1}) -\sum_{y\in S_{L-r-L_0}}\sum_{\substack{x\in\Lambda_{L,x}\\ \dist(x,y)=L_0-1}} \dprod{\delta_{x}}{(H^\omega_{\lambda,L,y}-z)^{-1}\delta_{x}}\biggr|\\
=\bigg|\sum_{y\in S_{L-r-L_0}}\sum_{\substack{x\in\Lambda_{L,x}\\ \dist(x,y)=L_0-1}}   \dprod{\delta_x}{(H^\omega_{\lambda,L}-z)^{-1}\delta_{Py}}\dprod{\delta_{y}}{(H^\omega_{\lambda,L,y}-z)^{-1}\delta_x}\bigg|.
\end{align*}
Next, using the estimate $|\dprod{\delta_x}{(H^\omega_{\lambda,L}-z)^{-1}\delta_y}|\leq \frac{1}{\Im z}$ we get,

\begin{align}\label{pf11Eq1}
  \begin{aligned}
    \bigg|C^\omega_{L,r}((\cdot-z)^{-1})&-\\
      &\hspace{-0.8in}\biggl.\frac{1}{(K+1)K^{L-r-L_0}}\sum_{y\in S_{L-r-L_0}}\frac{1}{K^{L_0-1}}\sum_{\substack{x\in\Lambda_{L,y}\\
          \dist(x,y)=L_0-1}}
      \dprod{\delta_{x}}{(H^\omega_{\lambda,L,y}-z)^{-1}\delta_{x}}\biggr| 
  \end{aligned}
\nonumber\\
  \leq \biggl(\frac{1}{(\Im z)^{2-s}}\frac{1}{(K+1)K^{L-r-L_0}}\biggr)\hspace{3in}\nonumber\\
  \sum_{y\in S_{L-r-L_0}}\frac{1}{K^{L_0-1}}\sum_{\substack{x\in\Lambda_{L,x}\\ \dist(x,y)=L_0-1}} |\dprod{\delta_{y}}{(H^\omega_{\lambda,L,y}-z)^{-1}\delta_x}|^s.
\end{align}

Now, observe that, the graphs $\Lambda_{L,y}$ are isomorphic to $\tilde{\Lambda}_{L_0+r-1}(\tilde{y})$, for any $\tilde{y}\in\mathcal{C}$ such that $\dist(\partial\mathcal{C},\tilde{y})=L_0+r-1$. 
So, using the independence of $H^\omega_{\lambda,L,y}$ and viewing them as cut-off operator for $H^\omega_{\mathcal{C},\lambda}$, we have
\begin{align*}
\lim_{L\rightarrow\infty}\frac{1}{|S_{L-r-L_0}|}\sum_{y\in S_{L-r-L_0}} \frac{1}{K^{L_0-1}}\sum_{\substack{x\in\Lambda_{L,y}\\ \dist(x,y)=L_0-1}}\dprod{\delta_x}{(H^\omega_{\lambda,L,y}-z)^{-1}\delta_x}\\
=\frac{1}{K^{L_0-1}}\sum_{\substack{x\in\tilde{\Lambda}_{L_0+r-1}(\tilde{y})\\ \dist(x,y)=L_0-1}}\Ex_\omega\left[ \dprod{\delta_x}{(\chi_{\tilde{\Lambda}_{L_0+r-1}(\tilde{y})}H^\omega_{\mathcal{C},\lambda} \chi_{\tilde{\Lambda}_{L_0+r-1}(\tilde{y})} -z)^{-1}\delta_x}\right],
\end{align*}
which follows through the first ergodic theorem, which gives almost sure convergence. Finally, using the symmetry of the tree $\tilde{\Lambda}_{L_0+r-1}(\tilde{y})$, we conclude that
\begin{align*}
 \frac{1}{K^{L_0-1}}\sum_{\substack{x\in\tilde{\Lambda}_{L_0+r-1}(\tilde{y})\\ \dist(x,y)=L_0-1}}\Ex_\omega\left[ \dprod{\delta_x}{(\chi_{\tilde{\Lambda}_{L_0+r-1}(\tilde{y})}H^\omega_{\mathcal{C},\lambda} \chi_{\tilde{\Lambda}_{L_0+r-1}(\tilde{y})} -z)^{-1}\delta_x}\right]\\
=\Ex_\omega\left[\dprod{\delta_{x_r}}{(\chi_{\tilde{\Lambda}_{L_0+r-1}(\tilde{y})}H^\omega_{\mathcal{C},\lambda} \chi_{\tilde{\Lambda}_{L_0+r-1}(\tilde{y})} -z)^{-1}\delta_{x_r}}\right],
\end{align*}
for some $x_r\in \tilde{\Lambda}_{L_0+r-1}(\tilde{y})$ such that
$\dist(\partial\mathcal{C},x_r)=r$. Using this in~\eqref{pf11Eq1}, and using~\eqref{exp}, we have,
\begin{align*}
 &\left|C^\omega_{L,r}((\cdot-z)^{-1})-\Ex_\omega\left[\dprod{\delta_{x_r}}{(\chi_{\tilde{\Lambda}_{L_0+r-1}(\tilde{y})}H^\omega_{\mathcal{C},\lambda} \chi_{\tilde{\Lambda}_{L_0+r-1}(\tilde{y})} -z)^{-1}\delta_{x_r}}\right]\right|\\
 &\qquad\leq \frac{1}{(\Im z)^{2-s}}\Ex\left[\left|\dprod{\delta_{\tilde{y}}}{(\chi_{\tilde{\Lambda}_{L_0+r-1}(y)}H^\omega_{\mathcal{C},\lambda}\chi_{\tilde{\Lambda}_{L_0+r-1}(\tilde{y})}-z)^{-1}\delta_{x_r}}\right|^s\right]\\
 &\qquad\leq \frac{1}{(\Im z)^{2-s}}C e^{-\gamma L_0}.
\end{align*}
Next, using the resolvent identity between $H^\omega_{\mathcal{C},\lambda}$ and $\chi_{\tilde{\Lambda}_{L_0+r-1}(\tilde{y})}H^\omega_{\mathcal{C},\lambda} \chi_{\tilde{\Lambda}_{L_0+r-1}(\tilde{y})}$ at $x_r$ and using~\eqref{exp}, we have
\begin{align*}
\left|\Ex_\omega\left[\dprod{\delta_{x_r}}{(\chi_{\tilde{\Lambda}_{L_0+r-1}(\tilde{y})}H^\omega_{\mathcal{C},\lambda} \chi_{\tilde{\Lambda}_{L_0+r-1}(\tilde{y})} -z)^{-1}\delta_{x_r}}-\dprod{x_r}{(H^\omega_{\mathcal{C},\lambda}-z)^{-1}\delta_{x_r}}\right]\right|\\
\leq \Ex\left[\left|\dprod{\delta_{x_r}}{(\chi_{\tilde{\Lambda}_{L_0+r-1}(\tilde{y})}H^\omega_{\mathcal{C},\lambda} \chi_{\tilde{\Lambda}_{L_0+r-1}(\tilde{y})} -z)^{-1}\delta_{\tilde{y}}}\dprod{\delta_{P\tilde{y}}}{(H^\omega_{\mathcal{C},\lambda}-z)^{-1}\delta_{x_r}}\right|\right]\\
\leq \frac{1}{(\Im z)^{2-s}}\Ex\left[\left|\dprod{\delta_{P\tilde{y}}}{(H^\omega_{\mathcal{C},\lambda}-z)^{-1}\delta_{x_r}}\right|^s\right]\leq \frac{C}{(\Im z)^{2-s}}e^{-\gamma L_0}.
\end{align*}
Combining the above results, and letting $L_0\rightarrow\infty$, we get
\begin{align*}
 \lim_{L\rightarrow\infty }C^\omega_{L,r}((\cdot-z)^{-1})=\Ex_\omega\left[\dprod{\delta_{x}}{(H^\omega_{\mathcal{C},\lambda}-z)^{-1}\delta_{x}}\right].
\end{align*}
This completes the proof of the theorem.
\subsection{Infinite divisibility and Compound Poisson Variables}
To show the infinite divisibility of the sequence of measures
$\{\mu^{\omega,\lambda}_{E_0,L}\}_L$, first, define the measures $\eta^{\omega,\lambda}_{E_0,L,x}$ for $x\in\Lambda_L$  as 
\begin{equation}\label{defInfDivMeaEq1}
 \eta^{\omega,\lambda}_{E_0,L,x}(f):=\tr(f(|\Lambda_L|(\chi_{\Lambda^\prime_{L-N}(x)}H^\omega_{\lambda,L} \chi_{\Lambda^\prime_{L-N}(x)}-E_0))),~ f\in C_c(\RR),
\end{equation}
where $N=\dist(0,x)$.

The following lemma says that the processes $\mu_{E_0,L}^{\omega,\lambda}$ and $\sum_{x: \dist(0,x) = l_L} \eta_{E_0, L,x}^{\omega,\lambda}$ have the same set of limit points in the topology of distributional convergence.  We have used the Fourier transform charactarization of the distributional convergence.
\begin{lem}\label{infDivLem1}
For $0<\alpha<\frac{1}{2}$, set $l_L=m_0\left\lfloor \frac{\alpha L}{m_0}\right\rfloor$. Then for $f\in C_c(\RR)$,
\begin{equation}\label{infDivResEq1}
\lim_{L\rightarrow\infty}\Ex_\omega\left[\left|e^{\iota\mu^{\omega,\lambda}_{E_0,L}(f)}-e^{\iota\sum_{x:\dist(0,x)=l_L}\eta^{\omega,\lambda}_{E_0,L,x}(f)}\right|\right]=0.
\end{equation}
\end{lem}
\begin{proof}
Using the decomposition $f=f_1-f_2+\iota (f_3-f_4)$ where $f_i\in
C_c(\RR)$ and positive, it is enough to show~\eqref{infDivResEq1} for
non-negetive functions. Therefore, we will assume  $f$ is a non-negetive function. Next, using the inequality $|e^{\iota x}-1|\leq |x|$, (pulling out one of the terms inside the modulus leaves us
with an expression of the form $ |e^{ix} -1| $) we only have to show
$$\lim_{L\rightarrow\infty}\Ex_\omega\Bigl[\Bigl|\mu^{\omega,\lambda}_{E_0,L}(f)-\sum_{x:\dist(0,x)=l_L}\eta^{\omega,\lambda}_{E_0,L,x}(f)\Bigr|\Bigr]=0.$$
Finally, since the functions $\Im(\cdot-z)^{-1}$ for $z\in\CC^{+}$ are in $C_0(\RR)$, it is enough to show  
\begin{equation}\label{infDivEstEq2}
\lim_{L\rightarrow\infty}\Ex_\omega\Bigl[\Bigl|\mu^{\omega,\lambda}_{E_0,L}(\Im(\cdot-z)^{-1})-\sum_{x:\dist(0,x)=l_L}\eta^{\omega,\lambda}_{E_0,L,x}(\Im(\cdot-z)^{-1})\Big|\Big]=0.
\end{equation}
 Hence, (denote $H^\omega_{\lambda,L,x}:=\chi_{\Lambda^\prime_{L-N}(x)}H^\omega_{\lambda,L} \chi_{\Lambda^\prime_{L-N}(x)}$)
\begin{align*}
&\Bigl|\mu^{\omega,\lambda}_{E_0,L}(\Im(\cdot-z)^{-1})-\sum_{x:\dist(0,x)=l_L}\eta^{\omega,\lambda}_{E_0,L,x}(\Im(\cdot-z)^{-1})\Bigr|\\
&=\frac{1}{|\Lambda_L|}\biggl|\sum_{y\in\Lambda_L}\Im \dprod{\delta_y}{(H^\omega_{\lambda,L}-E_0-|\Lambda_L|^{-1}z)^{-1}\delta_y}\Bigr.\\
&\qquad\qquad\qquad\biggl.-\sum_{x:\dist(0,x)=l_L}\sum_{y\in \Lambda_L^\prime(x)} \Im \dprod{\delta_y}{(H^\omega_{\lambda,L,x}-E_0-|\Lambda_L|^{-1}z)^{-1}\delta_y}\biggr|\\
&\leq \frac{1}{|\Lambda_L|}\sum_{\dist(0,x)<l_L}\Im G_L(y,y;z_L)\\
&\qquad +\frac{1}{|\Lambda_L|}\sum_{x:\dist(0,x)=l_L}\sum_{\substack{y\in \Lambda_L^\prime(x)\\ \dist(x,y)<l_L}} \Im G_L(y,y;z_L)+\Im G_{L,x}(y,y;z_L)\\
&\qquad +\frac{1}{|\Lambda_L|}\sum_{x:\dist(0,x)=l_L}\sum_{\substack{y\in \Lambda_L^\prime(x)\\ \dist(x,y)\geq l_L}} |G_L(y,y;z_L)-G_{L,x}(y,y;z_L)|,
\end{align*}
where $G_L(x,y;z_L)=\dprod{\delta_x}{(H^\omega_{\lambda,L}-E_0-|\Lambda_L|^{-1}z)^{-1}\delta_y}$ and $G_{L,x}(x,y;z_L)=\\ \langle\delta_x,(H^\omega_{\lambda,L,x}-E_0-|\Lambda_L|^{-1}z)^{-1}\delta_y\rangle$. Using~\eqref{WegnerEq3} for the first and the second sums, we get
\begin{align*}
&\frac{1}{|\Lambda_L|}\sum_{\dist(0,x)<l_L}\Ex_\omega\left[\Im G_L(y,y;z_L)\right]\leq \pi \norm{\rho}_\infty \frac{1+\frac{K+1}{K-1}(K^{l_L}-1)}{1+\frac{K+1}{K-1}(K^{L+1}-1)}\\
&\qquad\qquad\qquad\qquad\qquad\qquad\qquad\qquad=O(K^{-(1-\alpha)L})\xrightarrow{L\rightarrow\infty}0.\\
&\frac{1}{|\Lambda_L|}\sum_{x:\dist(0,x)=l_L}\sum_{\substack{y\in \Lambda_L^\prime(x)\\ \dist(x,y)<l_L}} \Ex_\omega\left[\Im G_L(y,y;z_L)+\Im G_{L,x}(y,y;z_L)\right]\\
&\qquad\qquad\leq 2\pi \norm{\rho}_\infty \frac{(K+1)K^{l_L}\frac{K^{l_L}-1}{K-1}}{1+\frac{K+1}{K-1}(K^{L+1}-1)}=O(K^{(1-2\alpha)L})\xrightarrow{L\rightarrow\infty}0.
\end{align*}
For the third term, we use the resolvent equation, and get ($Px$
denotes the neighboring vertex such that $\dist(0,x)=\dist(0,Px)+1$;
i.e., the vertex previous to $x$.)
\begin{align}
&\frac{1}{|\Lambda_L|}\sum_{x:\dist(0,x)=l_L}\sum_{\substack{y\in \Lambda_L^\prime(x)\\ \dist(x,y)\geq l_L}} \Ex[|G_L(y,y;z_L)-G_{L,x}(y,y;z_L)|]\nonumber\\
&\qquad= \frac{1}{|\Lambda_L|}\sum_{x:\dist(0,x)=l_L}\sum_{\substack{y\in \Lambda_L^\prime(x)\\ \dist(x,y)\geq l_L}} \Ex[|G_L(y,Px;z_L)G_{L,x}(x,y;z_L)|]\nonumber\\
&\qquad\leq \frac{1}{|\Lambda_L|}\sum_{x:\dist(0,x)=l_L}\sum_{\substack{y\in \Lambda_L^\prime(x)\\ \dist(x,y)\geq l_L}} \frac{1}{(|\Lambda_L|^{-1}\Im z)^{2-s}}\Ex[|G_{L,x}(x,y;z_L)|^s]\nonumber\\
&\qquad\leq \frac{|\Lambda_L|^{1-s}(K+1)K^{l_L-1}}{(\Im z)^{2-s}} \sum_{n=l_L}^\infty K^n e^{n(\tilde{C}-s\frac{1}{m_0}\ln\lambda)}\nonumber\\
&\qquad=O\left(e^{(1-s+2\alpha)\ln K L+\alpha L(\tilde{C}-s\frac{1}{m_0}\ln \lambda)}\right)\label{infDivEstEq1},
\end{align}
where we used the fact that $|G_L(x,y;z_L)|\leq
\frac{1}{|\Lambda_L|^{-1}\Im z}$, and the last expression comes from the proof of Theorem~\ref{LocResThm} (see~\eqref{expResEq1}). For
$$\frac{m_0((1-s+2\alpha)\ln K+\alpha \tilde{C})}{s\alpha}< \ln\lambda, $$
observe that~\eqref{infDivEstEq1} goes  to zero as $L\rightarrow\infty$. This completes the proof of~\eqref{infDivEstEq2}, and hence the lemma.
\end{proof}
Before attempting to prove the main result, we need to establish some results on the limit of the process. 
\begin{lem}\label{limMeaLem}
Let $L_n=m_0n$ for $n\in\NN$. Then, for any bounded interval $I$, 
\begin{equation}\label{limMeaEq1}
\lim_{n\rightarrow\infty}\Ex_\omega[\mu^{\omega,\lambda}_{E_0,L_n}(\chi_I)]=Kn_{\mathcal{C},\lambda}(E_0)|I|.
\end{equation}
Given any bounded interval $I$, there exists sub-sequence $\{\tilde{L}_n\}_n$ of $\{L_n\}_n$, such that $\{\sum_{\dist(0,x)=l_{\tilde{L}_m}}\mathbb{P}[\eta^{\omega,\lambda}_{E_0,\tilde{L}_{m},x}(\chi_I)=k]\}_{m\in\NN}$ converges.
Here $l_L$ is the sequence defined in Lemma~\ref{infDivLem1}.
\end{lem}
\begin{proof}
Using the Wegner estimate (equation~\eqref{WegnerEq2}) on~\eqref{meaSeqEq1}, we have,
$$\Ex_\omega[\mu^{\omega,\lambda}_{E_0,L}(\chi_I)]\leq C|I|$$
for any bounded interval $I$. So, the measure associated to the linear
functional $\Ex_\omega[\mu^{\omega,\lambda}_{E_0,L}(\cdot)]$ is
absolutely continuous with bounded density. From definition of
$\mu^{\omega,\lambda}_{E_0,L}$, (assume $L$ is divisible by $m_0$)
\begin{align}\label{limMeaPfEq1}
\Ex_\omega[\mu^{\omega,\lambda}_{E_0,L}(\chi_I)]&=\Ex\left[\tr(E_{H^\omega_{\lambda,L}}(E_0+|\Lambda_L|^{-1}I))\right]\nonumber\\
&=\sum_{x\in\Lambda_L}\Ex\left[\dprod{\delta_x}{E_{H^\omega_{\lambda,L}}(E_0+|\Lambda_L|^{-1}I)\delta_x}\right]\nonumber\\
&=\sum_{n=0}^{L}\sum_{\dist(x,\partial\Lambda_L)=n}\Ex\left[\dprod{\delta_x}{E_{H^\omega_{\lambda,L}}(E_0+|\Lambda_L|^{-1}I)\delta_x}\right]\nonumber\\
&=\sum_{n=0}^{L-1}K^{L-n}\Ex\left[\dprod{\delta_{x_n}}{E_{H^\omega_{\lambda,L}}(E_0+|\Lambda_L|^{-1}I)\delta_{x_n}}\right]\nonumber\\
&\qquad+\Ex\left[\dprod{\delta_{0}}{E_{H^\omega_{\lambda,L}}(E_0+|\Lambda_L|^{-1}I)\delta_{0}}\right],
\end{align}
where $x_n\in\Lambda_L$ is a vertex such that $\dist(x_n,\partial\Lambda_L)=n$ (for any $n$, there are multiple such vertex). 
Using~\eqref{WegnerEq1} on~\eqref{limMeaPfEq1} for $n>L_0$ where $L_0=\lfloor \alpha L\rfloor$ for some $0<\alpha<1$.
\begin{align}\label{limMeaPfEq2}
\Ex_\omega[\mu^{\omega,\lambda}_{E_0,L}(\chi_I)]&=\sum_{n=0}^{L_0}(K+1)K^{L-n}\Ex\left[\dprod{\delta_{x_n}}{E_{H^\omega_{\lambda,L}}(E_0+|\Lambda_L|^{-1}I)\delta_{x_n}}\right]\nonumber\\
&\qquad+O\left(K^{-L_0}|I|\right).
\end{align}
Therefore, to obtain the limit~\eqref{limMeaEq1}, we only need to
compute the limit of the RHS above, as $ L \rightarrow \infty $. Using the denseness of $\Im(\cdot-z)^{-1}$ for $z\in\CC^{+}$ in $C_0(\RR)$, we only have to find the limit of 
\begin{align}\label{limMeaPfEq3}
(K-1)\sum_{n=0}^{L_0}K^{-n}\Ex\left[\Im\dprod{\delta_{x_n}}{(H^\omega_{\lambda,L}-E_0-|\Lambda_L|^{-1}z)^{-1}\delta_{x_n}}\right],
\end{align}
as $ L \rightarrow \infty $, since as $L\rightarrow\infty$, we can replace $\frac{(K+1)K^{L-n}}{|\Lambda_L|}$ by $(K-1)K^{-n}$. 
This can be done because, the expectation is bounded, which follows from the proof of Lemma~\ref{lemWegner}.

Now, we continue as in the proof of Theorem~\ref{ThmDosCanopy}. Define
$$\tilde{H}^\omega_{\lambda,L}=\chi_{\Lambda_{m_0}(0)}H^\omega_{\lambda,L}\chi_{\Lambda_{m_0}(0)}+\sum_{\dist(0,y)=m_0+1}\chi_{\Lambda_{L,y}}H^\omega_{\lambda,L}\chi_{\Lambda_{L,y}},$$
where
$$\Lambda_{L,y}=\{x\in\Lambda_L:y\prec x\}.$$
Using the resolvent equation, for any $x$ with $\dist(x,\partial\Lambda_L)\leq L_0$, between $H^\omega_{\lambda,L}$ and $\tilde{H}^\omega_{\lambda_L}$ (we will denote $z_L=E_0+|\Lambda_L|^{-1}z$),
\begin{align}\label{limMeaPfEq4}
&\dprod{\delta_x}{(H^\omega_{\lambda,L}-z_L)^{-1}\delta_x}=\dprod{\delta_x}{(H^\omega_{\lambda,L,A_x}-z_L)^{-1}\delta_x}-\nonumber\\
&\hspace{2in}\dprod{\delta_x}{(H^\omega_{\lambda,L}-z_L)^{-1}\delta_{PA_x}}\dprod{\delta_{A_x}}{(H^\omega_{\lambda,L,A_x}-z_L)^{-1}\delta_x},
\end{align}
where $A_x$ is the unique vertex satisfying $\dist(0,A_x)=m_0+1$ and $x\in\Lambda_{L,A_x}$,
 $H^\omega_{\lambda,L,y}=\chi_{\Lambda_{L,y}}H^\omega_{\lambda,L}\chi_{\Lambda_{L,y}}$,
 and $PA_x\in\Lambda_{m_0}(0)$ is such that $\dist(A_x,PA_x)=1$. Using~\eqref{limMeaPfEq4} on~\eqref{limMeaPfEq3} will give
\begin{align}\label{limMeaPfEq5}
&(K-1)\sum_{n=0}^{L_0}K^{-n}\Ex\left[\Im\dprod{\delta_{x_n}}{(H^\omega_{\lambda,L}-z_L)^{-1}\delta_{x_n}}\right]=\nonumber\\
&(K-1)\sum_{n=0}^{L_0} K^{-n}\Ex\left[\Im\dprod{\delta_{x_n}}{(H^\omega_{\lambda,L,A_{x_n}}-z_L)^{-1}\delta_{x_n}}\right]-\nonumber\\
&(K-1)\sum_{n=0}^{L_0} K^{-n}\Ex\left[\Im \left(\dprod{\delta_{x_n}}{(H^\omega_{\lambda,L}-z_L)^{-1}\delta_{PA_{x_n}}}\dprod{\delta_{A_{x_n}}}{(H^\omega_{\lambda,L,A_{x_n}}-z_L)^{-1}\delta_{x_n}}\right)\right].
\end{align}
Embedding the graph $\Lambda_{L,A_x}$ into $\mathcal{C}$ and using the resolvent identity, we get
\begin{align*}
&\dprod{\delta_x}{(H^\omega_{\lambda,L,A_x}-z_L)^{-1}\delta_x}=\\
&\quad\quad \dprod{\delta_x}{(H^\omega_{\mathcal{C},\lambda}-z_L)^{-1}\delta_x}+\dprod{\delta_x}{(H^\omega_{\lambda,L,A_x}-z_L)^{-1}\delta_{A_x}}\dprod{\delta_{PA_x}}{(H^\omega_{\mathcal{C},\lambda}-z_L)^{-1}\delta_x},
\end{align*}
which can be used on~\eqref{limMeaPfEq5} to give
\begin{align*}\label{limMeaPfEq6}
&(K-1)\sum_{n=0}^{L_0}K^{-n}\Ex\left[\Im\dprod{\delta_{x_n}}{(H^\omega_{\lambda,L}-z_L)^{-1}\delta_{x_n}}\right]=\nonumber\\
& (K-1)\sum_{n=0}^{L_0}K^{-n}\Ex\left[\Im\dprod{\delta_{x_n}}{(H^\omega_{\mathcal{C},\lambda}-z_L)^{-1}\delta_{x_n}}\right]+\nonumber\\
&(K-1)\sum_{n=0}^{L_0}K^{-n}\Ex\Bigl[\Bigr.\Im\Bigl(\bigr.\dprod{\delta_{x_n}}{(H^\omega_{\lambda,L,A_{x_n}}-z_L)^{-1}\delta_{A_{x_n}}}\times\\
&\qquad\qquad\qquad\qquad\qquad\qquad\qquad\qquad\qquad\dprod{\delta_{PA_{x_n}}}{(H^\omega_{\mathcal{C},\lambda}-z_L)^{-1}\delta_{x_n}}\Bigl.\Bigr)\Bigl.\Bigr]-\nonumber\\
&(K-1)\sum_{n=0}^{L_0} K^{-n}\Ex\left[\Im \left(\dprod{\delta_{x_n}}{(H^\omega_{\lambda,L}-z_L)^{-1}\delta_{PA_{x_n}}}\dprod{\delta_{A_{x_n}}}{(H^\omega_{\lambda,L,A_{x_n}}-z_L)^{-1}\delta_{x_n}}\right)\right].
\end{align*}
Using the exponential decay estimate~\eqref{expResEq1}, the second and
third terms give
\begin{align*}
&(K-1)\sum_{n=0}^{L_0}K^{-n}\Ex\left[\left|\dprod{\delta_{x_n}}{(H^\omega_{\lambda,L,A_{x_n}}-z_L)^{-1}\delta_{A_{x_n}}}\dprod{\delta_{PA_{x_n}}}{(H^\omega_{\mathcal{C},\lambda}-z_L)^{-1}\delta_{x_n}}\right|\right]+\\
&(K-1)\sum_{n=0}^{L_0} K^{-n}\Ex\left[\left|\dprod{\delta_{x_n}}{(H^\omega_{\lambda,L}-z_L)^{-1}\delta_{PA_{x_n}}}\dprod{\delta_{A_{x_n}}}{(H^\omega_{\lambda,L,A_{x_n}}-z_L)^{-1}\delta_{x_n}}\right|\right]\\
&\leq  (K-1)\sum_{n=0}^{L_0} K^{-n}2\left(\frac{\tilde{C}_{m_0,\mu}}{\lambda^s}\right)^{\frac{L-L_0}{m_0}}\frac{|\Lambda_L|^{2-s}}{(\Im z)^{2-s}}\\
&\leq O\left(e^{\frac{L}{m_0}(m_0(2-s)\ln K+(1-\alpha)(\ln \tilde{C}_{m_0,\mu}-s\ln \lambda))}\right).
\end{align*}
Hence, for 
$$\frac{m_0(2-s)\ln K +(1-\alpha)\ln \tilde{C}_{\mu,m_0}}{s(1-\alpha)}<\ln \lambda,$$
we get
$$\lim_{L\rightarrow\infty}\Ex[\mu^{\omega,\lambda}_{E_0,L}(\Im(\cdot-z)^{-1})]=\lim_{L\rightarrow\infty}(K-1)\sum_{n=0}^{\lfloor \alpha L\rfloor}K^{-n}\Ex[\Im\dprod{\delta_{x_n}}{(H^\omega_{\mathcal{C},\lambda}-z_L)^{-1}\delta_{x_n}}].$$
But the RHS converges to $Kn_{\mathcal{C},\lambda}(E_0)$, where
$n_{\mathcal{C},\lambda}(E_0)$  is the density of the measure
$n_{\mathcal{C},\lambda}$ (which is the density of state measure; see
Theorem~\ref{ThmDosCanopy} for the definition), at $E_0$.

For the second assertion of the lemma, using Lemmas~\ref{infDivLem1} and~\eqref{limMeaEq1}, we have
\[\lim_{n\rightarrow\infty} \sum_{\dist(0,x)=l_{nm_0}}\Ex[\eta^{\omega,\lambda}_{E_0,nm_0,x}(\chi_{I})]=\lim_{n\rightarrow\infty}\Ex[\mu^{\omega,\lambda}_{E_0,nm_0}(\chi_I)]=Kn_{\mathcal{C},\lambda}(E_0)|I|.\]
Using 
$$0\leq \sum_{\dist(0,x)=l_{L}}\mathbb{P}[\eta^{\omega,\lambda}_{E_0,L,x}(\chi_{I})=k]\leq \frac{1}{k}\sum_{\dist(0,x)=l_{L}}\Ex[\eta^{\omega,\lambda}_{E_0,L,x}(\chi_{I})],$$
we can find a subsequence $\{\tilde{L}_m\}_m$ of $\{nm_0\}_{n\in\NN}$, such that \[\biggl\{\sum_{\dist(0,x)=l_{L}}\mathbb{P}[\eta^{\omega,\lambda}_{E_0,L,x}(\chi_{I})=k]\biggr\}\] converges and the limit lies in the interval $[0,\frac{K}{k}n_{\mathcal{C},\lambda}(E_0)|I|]$.
\end{proof}

\subsection*{Proof of Theorem~\ref{mainThm}}
To prove the theorem all we need to do is to compute
$$\lim_{L\rightarrow\infty}\Ex_\omega\left[e^{\iota t\mu^{\omega,\lambda}_{E_0,L}(\chi_I)}\right],$$
where $\chi_I$ is the characteristic function of a bounded interval $I\subset \RR$. Notice that, for $\chi_I$, the random variables $\mu^{\omega,\lambda}_{E_0,L}(\chi_I)$ and $\{\eta^{\omega,\lambda}_{E_0,L,x}(\chi_I)\}_x$ are integer valued.

Using the Lemma~\ref{infDivLem1}, we have,
\begin{align}\label{chaEqua1}
\lim_{L\rightarrow\infty}\Ex_\omega\left[e^{\iota t\mu^{\omega,\lambda}_{E_0,L}(\chi_I)}\right]&=\lim_{L\rightarrow\infty}\Ex_\omega\left[e^{\iota t\sum_{\dist(0,x)=l_L}\eta^{\omega,\lambda}_{E_0,L,x}(\chi_I)}\right]\nonumber\\
&=\lim_{L\rightarrow\infty}\prod_{\dist(0,x)=l_L }\left(\Ex_\omega\left[e^{\iota t\eta^{\omega,\lambda}_{E_0,L,x}(\chi_I)}\right]\right)\nonumber\\
&=\lim_{L\rightarrow\infty}e^{\sum_{\dist(0,x)=l_L} \ln\left(\Ex_\omega\left[e^{\iota t\eta^{\omega,\lambda}_{E_0,L,x}(\chi_I)}-1\right]+1\right)}.
\end{align}
The second line follows because of the independence of
$\{\eta^{\omega,\lambda}_{E_0,L,x}\}_{\dist(0,x)=l_L}$ (this is where
$l_L$ is a multiple of $m_0$ is used so that all of the the
projections $P_{p_j}$ has support at atmost one $\{\Lambda^\prime_{L-l_L}(x)\}_{x}$). Using $|e^{\iota x}-1|\leq |x|$ for real $x$, we have,
\begin{equation}\label{mainPfeq1}
 \Bigl|\Ex_\omega\left[e^{\iota t\eta^{\omega,\lambda}_{E_0,L,x}(\chi_I)}-1\right]\Bigr|\leq \Ex_\omega\left[\left|e^{\iota t\eta^{\omega,\lambda}_{E_0,L,x}(\chi_I)}-1\right|\right]\leq |t|\Ex_\omega\left[\eta^{\omega,\lambda}_{E_0,L,x}(\chi_I)\right],
\end{equation}
and using the Wegner estimate~\eqref{WegnerEq2} over the operator $\chi_{\Lambda^\prime_{L-l_L}(x)}H^\omega_{\lambda,L}\chi_{\Lambda^\prime_{L-l_L}(x)}$, (the measure $\eta^{\omega,\lambda}_{E_0,L,x}$ is defined using this, see~\eqref{defInfDivMeaEq1}) we get,
\begin{equation}\label{mainPfeq2}
 \left|\Ex_\omega\left[e^{\iota t\eta^{\omega,\lambda}_{E_0,L,x}(\chi_I)}-1\right]\right|\leq |t|\Ex_\omega\left[\eta^{\omega,\lambda}_{E_0,L,x}(\chi_I)\right]\leq C |t||I|\frac{|\Lambda^\prime_{L-l_L}(x)|}{|\Lambda_L|} \xrightarrow{L\rightarrow\infty}0.
\end{equation}
From the expressions~\eqref{mainPfeq1},~\eqref{mainPfeq2}, and the fact that $|\ln(1+x)-x|\leq |x|^2$ for $|x|\ll 1$, we have
\begin{align}\label{mainPfeq3}
&\ln\left(\Ex_\omega\left[e^{\iota t\eta^{\omega,\lambda}_{E_0,L,x}(\chi_I)}-1\right]+1\right)\nonumber\\
&\qquad=\Ex_\omega\left[e^{\iota t\eta^{\omega,\lambda}_{E_0,L,x}(\chi_I)}-1\right]+O\left(\left|\Ex_\omega\left[e^{\iota t\eta^{\omega,\lambda}_{E_0,L,x}(\chi_I)}-1\right]\right|^2\right)\nonumber\\
&\qquad=\Ex_\omega\left[e^{\iota t\eta^{\omega,\lambda}_{E_0,L,x}(\chi_I)}-1\right]+O\left(|t|^2\left(\Ex_\omega\left[\eta^{\omega,\lambda}_{E_0,L,x}(\chi_I)\right]\right)^2\right)\nonumber\\
&\qquad=\Ex_\omega\left[e^{\iota t\eta^{\omega,\lambda}_{E_0,L,x}(\chi_I)}-1\right]+O\left(|t|^2|I|^2{\left(\frac{|\Lambda^\prime_{L-l_L}(x)|}{|\Lambda_L|}\right)}^2\right).
\end{align}
Using~\eqref{mainPfeq3} and the fact that $|\Lambda_{L-l_L}(0)|\left(\frac{|\Lambda^\prime_{L-l_L}(x)|}{|\Lambda_L|}\right)^2\xrightarrow{L\rightarrow\infty} 0$ on~\eqref{chaEqua1}, we get
\begin{align}\label{chaEqua2}
 \lim_{L\rightarrow\infty}\Ex_\omega\left[e^{\iota t\mu^{\omega,\lambda}_{E_0,L}(\chi_I)}\right]&=\lim_{L\rightarrow\infty}e^{\sum_{\dist(0,x)=l_L} \ln\left(\Ex_\omega\left[e^{\iota t\eta^{\omega,\lambda}_{E_0,L,x}(\chi_I)}-1\right]+1\right)}\nonumber\\
 &=\lim_{L\rightarrow\infty}e^{\sum_{\dist(0,x)=l_L}
   \Ex_\omega\left[e^{\iota
   t\eta^{\omega,\lambda}_{E_0,L,x}(\chi_I)}-1\right]}\ .
\end{align}
Focusing on the exponent of the last equation, we have
\begin{align}\label{mainPfeq4}
&\sum_{\dist(0,x)=l_L} \Ex_\omega\left[e^{\iota t\eta^{\omega,\lambda}_{E_0,L,x}(\chi_I)}-1\right]\nonumber\\
&\qquad=\sum_{\dist(0,x)=l_L} \sum_{k=1}^\infty (e^{\iota tk}-1)\mathbb{P}[\eta^{\omega,\lambda}_{E_0,L,x}(\chi_I)=k]\nonumber\\
&\qquad= \sum_{k=1}^{M_0} (e^{\iota tk}-1)\Bigl(\sum_{\dist(0,x)=l_L} \mathbb{P}[\eta^{\omega,\lambda}_{E_0,L,x}(\chi_I)=k]\Bigr)+R(L),
\end{align}
where 
\begin{align}\label{mainPfeq5}
R(L)&= \sum_{\dist(0,x)=l_L} \sum_{k=M_0+1}^{\infty} (e^{\iota tk}-1)\mathbb{P}[\eta^{\omega,\lambda}_{E_0,L,x}(\chi_I)=k]\nonumber\\
|R(L)|&\leq 2\sum_{\dist(0,x)=l_L} \sum_{k=M_0+1}^{\infty} \mathbb{P}[\eta^{\omega,\lambda}_{E_0,L,x}(\chi_I)=k]\nonumber\\
&\leq (\pi\norm{\rho}_\infty |I|)^2 |\Lambda_{L-l_L}(0)|\left(\frac{|\Lambda^\prime_{L-l_L}(x)|}{|\Lambda_L|}\right)^2\xrightarrow{L\rightarrow\infty}0.
\end{align}
The last line follows from the Minami estimate~\eqref{MinamiEq1}. From Lemma~\ref{limMeaLem}, we have a sequence
$\{L_n\}_{n\in\NN}$, such that, for $ k\leq M_0$
\begin{align}\label{mainPfeq6}
 &\lim_{n\rightarrow\infty}\sum_{\dist(0,x)=l_{L_n}} \mathbb{P}[\eta^{\omega,\lambda}_{E_0,L_n,x}(\chi_I)=k]=p_k(I).\\
\intertext{ Combining~\eqref{mainPfeq6},~\eqref{mainPfeq5}, and~\eqref{chaEqua2} gives}
&\lim_{n\rightarrow\infty}\Ex_\omega\left[e^{\iota
  t\mu^{\omega,\lambda}_{E_0,L}(\chi_I)}\right]=e^{\sum_{k=1}^{M_0}(e^{\iota
  t k}-1)p_k(I)}.
\end{align}
This completes the proof of the theorem.
\section*{Acknowledgement}
I thank my guide Prof. M Krishna for the guidance and support he is
giving me. I thank Dhriti Ranjan Dolai for being the understanding and
encouraging person he is, and Anish Mallick for all the numerous helps
he has done at each stage of this work; it would have been impossible to complete this work
with out him.

\bibliographystyle{plain}

\end{document}